\documentclass[a4paper,11pt]{amsart}

\usepackage[french, ngerman, italian, english]{babel}
\usepackage{amsmath,amsthm, amssymb}
\usepackage{setspace}
\usepackage{anysize}
\usepackage{url}
\usepackage{bm}
\usepackage{graphicx}
\usepackage[colorlinks=true]{hyperref}
\usepackage{mathptmx}
\usepackage{paralist}
\usepackage{verbatim}
\usepackage[utf8]{inputenc}

\theoremstyle{plain}
\newtheorem{thm}{\bf Theorem}[section]

\newtheorem{prop}[thm]{\bf Proposition}
\newtheorem{lemma}[thm]{\bf Lemma}
\newtheorem{corollary}[thm]{\bf Corollary}
\newtheorem{question}[thm]{\bf Question}

\theoremstyle{definition}
\newtheorem{definition}[thm]{\bf Definition}
\theoremstyle{remark}
\newtheorem{remark}[thm]{\bf Remark}
\newtheorem{discussion}[thm]{\bf Discussion}
\newtheorem{example}[thm]{\bf Example}
\theoremstyle{example}

\def\NZQ{\Bbb}               % the font for N,Z,Q,R,C
\def\NN{{\NZQ N}}

\def\ZZ{{\NZQ Z}}

\def\PP{{\NZQ P}}

\def\frk{\frak}               % font for "Fraktur"

\def\mm{{\frk m}}

\def\Phi{{\frk n}}
\def\Phi{{\frk N}}
\def\opn#1#2{\def#1{\operatorname{#2}}} % to make operators
\opn\chara{char} \opn\length{\ell} \opn\pd{pd} \opn\rk{rk}
\opn\projdim{proj\,dim} \opn\injdim{inj\,dim} \opn\rank{rank}
\opn\depth{depth} \opn\grade{grade} \opn\height{height}
\opn\embdim{emb\,dim} \opn\codim{codim}

\opn\Tr{Tr} \opn\bigrank{big\,rank}
\opn\superheight{superheight}\opn\lcm{lcm}
\opn\trdeg{tr\,deg}%\emph{
\opn\reg{reg} \opn\lreg{lreg} \opn\ini{in} \opn\lpd{lpd}
\opn\size{size} \opn\sdepth{sdepth}
\opn\link{link}\opn\fdepth{fdepth}\opn\lex{lex}
\opn\div{div} \opn\Div{Div} \opn\cl{cl} \opn\Cl{Cl}
\opn\Spec{Spec} \opn\Supp{Supp} \opn\supp{supp} \opn\Sing{Sing}
\opn\Ass{Ass} \opn\Min{Min}\opn\Mon{Mon}
\opn\Ann{Ann} \opn\Rad{Rad} \opn\Soc{Soc}
\opn\Im{Im} \opn\Ker{Ker} \opn\Coker{Coker} \opn\Am{Am}
\opn\Hom{Hom} \opn\Tor{Tor} \opn\Ext{Ext} \opn\End{End}
\opn\Aut{Aut} \opn\id{id}

\opn\nat{nat}
\opn\pff{pf}%   \pf exists already
\opn\Pf{Pf} \opn\GL{GL} \opn\SL{SL} \opn\mod{mod} \opn\ord{ord}
\opn\Gin{Gin} \opn\Hilb{Hilb}\opn\sort{sort}
\opn\aff{aff} \opn\con{conv} \opn\relint{relint} \opn\st{st}
\opn\lk{lk} \opn\cn{cn} \opn\core{core} \opn\vol{vol}
\opn\link{link} \opn\star{star}\opn\lex{lex}\opn\set{set}
\opn\gr{gr}

\def\pot#1#2{#1[\kern-0.28ex[#2]\kern-0.28ex]}

\opn\dirlim{\underrightarrow{\lim}}
\opn\inivlim{\underleftarrow{\lim}}
\def \chara{{\operatorname{char}}}

\def \height{{\operatorname{ht}}}
\def \reg{{\operatorname{reg}}}
\def \depth{{\operatorname{depth}}}
\def \Gin{{\operatorname{Gin}}}
\def \grade{{\operatorname{grade}}}
\def \Ker{{\operatorname{Ker}}}
\def \pd{{\operatorname{pd}}}
\def \mm{{\mathfrak{m}}}

\def \NN{\mathbb N}
\def \ZZ{\mathbb Z}

\def \M{\mathcal M}
\def \S{\mathcal S}

\def \chara{{\operatorname{char}}}

\def \height{{\operatorname{ht}}}
\def \reg{{\operatorname{reg}}}
\def \depth{{\operatorname{depth}}}
\def \Gin{{\operatorname{Gin}}}
\def \grade{{\operatorname{grade}}}
\def \Ker{{\operatorname{Ker}}}
\def \pd{{\operatorname{pd}}}
\def \mm{{\mathfrak{m}}}

\def \NN{\mathbb N}
\def \ZZ{\mathbb Z}

\def \vv{{\bf v}}

\def \iii{{\bf i}}
\def \jjj{{\bf j}}

\def \S_d{\mathcal{M}(d)}
\def \M{\mathcal M}
\def \S{\mathcal S}
\def \SS{\mathbb S}
\def \TT{\mathbb T}

\newcommand{\m}{\frak{m}}

\def\ini{\operatorname{\rm in}}

\begin{document}
\title{The possible extremal Betti numbers of a homogeneous ideal}
\author{J\"urgen Herzog}
\address{Fachbereich Mathematik, Universit\"at Duisburg-Essen, Campus Essen, 45117
Essen, Germany}
\email{juergen.herzog@uni-essen.de}
\author{Leila Sharifan}
\address{Department of Mathematics, Hakim Sabzevari University, Sabzevar, Iran}
\email{leila-sharifan@aut.ac.ir}
\author{Matteo Varbaro}
\address{Dipartimento di Matematica,
Universit\`a degli Studi di Genova, Italy}
\email{varbaro@dima.unige.it}
\date{{\small \today}}
\maketitle

\begin{abstract}
We give a numerical characterization of the possible extremal Betti numbers (values as well as positions) of any homogeneous ideal in a polynomial ring over a field of characteristic $0$.
\end{abstract}

\section*{Introduction}
The main purpose of this note is to characterize the possible extremal Betti numbers (values as well as positions) of any homogeneous ideal in a polynomial ring over a field of characteristic $0$. These special graded Betti numbers were introduced by Bayer, Charalambous and Popescu in \cite{BCP}: One reason why they naturally arise is that they are equal to the dimensions of certain cohomology vector spaces of the projective scheme associated to the ideal. We also investigate on the possible Betti tables of a componentwise linear ideal: Such a problem seems to be very hard, indeed we could solve it just in some special cases. We provide some examples illustrating the main obstructions to the issue.

\vspace{2mm}

After a preliminary section, in Section \ref{seccomp} we  study the possible Betti tables of componentwise linear ideals, introduced by Herzog and Hibi in \cite{HH}. This issue is equivalent to a characterization of the graded Betti numbers of strongly stable ideals. We denote by $I_j$  the $j$th graded component of a strongly stable ideal $I$,  set $\mu_{ij}(I)$ equal to the number of the monomials in $I_{j}\cap K[x_1,\ldots ,x_i]$ divisible by $x_i$, and define the matrix $\M(I)=(\mu_{ij}(I))$, which we call the {\it matrix of generators} of $I$. As explained in the beginning of Section \ref{seccomp}, the matrix $\M(I)$ and the graded Betti numbers of $I$ determine each other. Thus we are led to characterize the integer matrices $(\mu_{ij})$ for which there exists a strongly stable ideal $I$ such that  $\M(I)=(\mu_{ij})$. From a result obtained by Murai in \cite{Mu}, which yields the characterization of the possible Betti numbers of ideals with linear resolution (see Proposition \ref{mainlinear}), one can deduce some necessary conditions for $(\mu_{ij})$ being the matrix of generators for some strongly stable ideal, Proposition~\ref{main}. Unfortunately these conditions are not sufficient to describe the matrices of generators of  strongly stable ideals, as shown in Example \ref{obstructionscomplin}. The difficulty of the task of characterizing Betti tables of componentwise linear ideals is also shown by Example \ref{nagelroemer}, where we exhibit a noncomponentwise linear ideal (in three variables) with the same Betti table of a componentwise linear ideal, answering negatively a question raised by Nagel and R\"omer in \cite{NR}. After discussing the main obstruction to constructing strongly stable ideals with a prescribed matrix of generators, we give sufficient conditions for a matrix to be of the form $\M(I)$ where $I$ is strongly stable in Proposition \ref{suffcomplin}. As a consequence it is shown in Corollary~\ref{3variables} that the necessary conditions given in Proposition~\ref{main} are also sufficient when dealing with strongly stable ideals in three variables. Another instance for which  the matrix of generators of a particular class of strongly stable ideals can be described is given in Proposition~\ref{d-lex}, which gives the possible matrices of generators of lexsegment ideals, and therefore, exploiting a result in \cite{HH}, of Gotzmann ideals.
Though a complete characterization of the possible Betti numbers of a strongly stable ideal seems to be quite difficult, we succeed in Section \ref{secextremal} to characterize all possible extremal Betti numbers of any homogeneous ideal $I\subset S=K[x_1,\ldots ,x_n]$, provided that $K$ has characteristic $0$. According to \cite{BCP}, a Betti number $\beta_{i, i+j}\neq 0$ of $I$ is called extremal if $\beta_{k, k+l}=0$ for all pairs $(k,l)\neq (i,j)$ with  $k\geq i$ and  $l\geq j$. It is  shown in \cite{BCP} that the positions as well as the values  of the extremal Betti numbers of a graded ideal are preserved under taking the generic initial ideal with respect to the reverse lexicographical order. Thus, since $K$ has characteristic $0$, we may restrict our attention to characterize the extremal Betti numbers of strongly stable ideals. 
More precisely, let $i_1<i_2<\cdots <i_k<n$, $j_1>j_2> \cdots >j_k$ and $b_1,\ldots,b_k$ be sequences of positive integers. In Theorem~\ref{extremalthm} we give  numerical conditions which are  equivalent to the property that there exists a homogeneous ideal $I\subset S$ whose  extremal Betti numbers are precisely  $\beta_{i_p,i_p+j_p}(I)=b_p$ for $p=1,\ldots, k$.

We are very grateful to the anonymous referee for suggesting us the point (iv) of Theorem \ref{extremalthm} and the last statement in Lemma \ref{mireferee1}.

\section{Preliminaries}

Let $n$ be a positive integer. We will essentially work with the polynomial ring
\[S=K[x_1,\ldots ,x_n],\]
where the $x_i$'s are variables over a field $K$. 
Given a monomial $u\in S$, we set:
\begin{equation}\label{maximumvariable}
m(u)=\max \{e\in \{1,\ldots ,n\} \ : \ x_e \ \mbox{ divides } \ u\}.
\end{equation}
A monomial space $V\subset S$ is called {\it stable} if for any monomial $u\in V$, then $(u/x_{m(u)})\cdot x_i\in V$ for all $i< m(u)$. It is called {\it strongly stable} if for any monomial $u\in V$ and for each $j\in \{1,\ldots ,n\}$ such that $x_j$ divides $u$, then $(u/x_j)\cdot x_i\in V$ for all $i< j$. Obviously a strongly stable monomial space is stable. By a stable (strongly stable)  monomial ideal $I\subseteq S$ we mean that the underlying monomial space is stable (strongly stable); or equivalently, that the monomial space $\langle G(I)\rangle$, where $G(I)$ is the unique minimal set of monomial generators of $I$, is stable (strongly stable).

On the monomials, unless we explicitly say differently, we use a degree lexicographical order with respect to the ordering of the variables $x_1>x_2>\ldots >x_n$. A monomial space $V\subset S$ is called {\it lexsegment} if, for all $d\in \NN$, there exists a monomial $u\in S_d$ (the degree $d$ part of $S$) such that
\[V\cap S_d=\langle v\in S_d:v\geq u\rangle.\]
We will sometimes denote by:
\[L_{\ge u}=\{v\in S_d:v\geq u\}.\]
Clearly, a lexsegment monomial space is strongly stable. The celebrated theorem of Macaulay explains when a lexsegment monomial space is an ideal. We remind that given a natural number $a$ and a positive integer $d$, the $d$th {\it Macaulay representation} of $a$ is the unique writing:
\[a=\sum_{i=1}^d\binom{k(i)}{i} \ \ \ \mbox{such that } \ k(d)>k(d-1)>\ldots>k(1)\geq 0,\]
see \cite[Lemma 4.2.6]{BH}. Then:
\[a^{\langle d\rangle}=\sum_{i=1}^d\binom{k(i)+1}{i+1}.\]
A numerical sequence $(h_i)_{i\in\NN}$ is called an {\it O-sequence} if $h_0=1$ and $h_{d+1}\leq h_d^{\langle d\rangle}$ for all $d\geq 1$. (The reader should be careful because the definition of $O$-sequence depends on the numbering: A vector $(m_1,\ldots ,m_n)$ will be an $O$-sequence if $m_1=1$ and and $m_{i+1}\leq m_i^{\langle i-1\rangle}$ for all $i\geq 2$). The theorem of Macaulay (for example see \cite[Theorem 4.2.10]{BH}) says that, given a numerical sequence $(h_i)_{i\in\NN}$, the following are equivalent:
\begin{itemize}
\item[(i)] $(h_i)_{i\in\NN}$ is an {\it O}-sequence with $h_1\leq n$.
\item[(ii)] There is a homogeneous ideal $I\subset S$ such that $(h_i)_{i\in\NN}$ is the Hilbert function of $S/I$.
\item[(iii)] The lexsegment monomial space $L\subset S$ such that $L\cap S_d$ consists in the biggest $\binom{n+d-1}{d}-h_d$ monomials, is an ideal.
\end{itemize}

For any $\ZZ$-graded finitely generated $S$-module $M$, there is a minimal graded free resolution:
\[0\rightarrow \bigoplus_{j\in \ZZ}S(-j)^{\beta_{p,j}(M)}\rightarrow \bigoplus_{j\in \ZZ}S(-j)^{\beta_{p-1,j}(M)}\rightarrow \ldots \rightarrow \bigoplus_{j\in \ZZ}S(-j)^{\beta_{0,j}(M)}\rightarrow M \rightarrow 0,\]
where $S(k)$ denotes the $S$-module $S$ supplied with the new grading $S(k)_i=S_{k+i}$. The celebrated Hilbert's Syzygy theorem (for example see \cite[Corollary 2.2.14 (a)]{BH}) guarantees $p\leq n$. The natural numbers $\beta_{i,j}=\beta_{i,j}(M)$ are numerical invariants of $M$, and they are called the {\it graded Betti numbers} of $M$. The coarser invariants $\beta_i=\beta_i(M)=\sum_{j\in \ZZ}\beta_{i,j}$ are called the {\it (total) Betti numbers} of $M$. A graded Betti number $\beta_{i,i+d}$ is said to be {\it extremal} if $\beta_{i,i+d}\neq 0$ and for all $(p,q)\neq (i,d)$ such that $p\geq i$ and $q\geq d$, $\beta_{p,p+q}=0$. We will refer to the matrix $(\beta_{i,j})$ as the {\it Betti table} of $M$.  Actually, in the situations we will consider in this paper $M=I$ is a homogeneous ideal of $S$. In this case $\beta_{i,j}=0$ whenever $i\geq n$ or $j\leq i$ (unless $I=S$). We will present the Betti table of $I$ as follows:
\[
\left(\begin{array}{cccccc}
                     \beta_{0,1} & \beta_{1,2} & \beta_{2,3} & \cdots & \cdots & \beta_{n-1,n}\\
                     \beta_{0,2} & \beta_{1,3} & \beta_{2,4} & \cdots & \cdots & \beta_{n-1,n+1}\\
                     \beta_{0,3} & \beta_{1,4} & \beta_{2,5} & \cdots & \cdots & \beta_{n-1,n+2}\\
                     \vdots & \vdots & \vdots & \cdots & \cdots & \vdots \\
                   \end{array}\right).
\]
Also if the definition of the Betti table of $M$ predicts infinite many rows, only a finite number of entries are not zero (because $M$ is finitely generated). Consequently, only a finite number of rows are significant, and in the examples we will present throughout the paper we will draw just the significant rows. Notice that a graded Betti number is extremal if and only if it is the nonzero top left ``corner" in a block of zeroes of the Betti table.

For an integer $d$, the $S$-module $M$ is said to have a {\it $d$-linear resolution} if $\beta_{i,j}(M)=0$ for every $i=0,\ldots ,p$ and $j\neq i+d$; equivalently, if $\beta_i(M)=\beta_{i,i+d}(M)$ for any $i=0,\ldots ,p$. Notice that if $M$ has $d$-linear resolution, then it is generated in degree $d$. The $S$-module $M$ is said {\it componentwise linear} if $M_{\langle d\rangle}$ has $d$-linear resolution for all $d\in\ZZ$, where $M_{\langle d\rangle}$ means the $S$-submodule of $M$ generated by the elements of degree $d$ of $M$. It is not difficult to show that if $M$ has a linear resolution, then it is componentwise linear.

Let $I$ be a stable monomial ideal. For all $i\in\{1,\ldots ,n\}$ and $d\in\NN$ we set:
\begin{eqnarray}\label{defmid}
m_{i,d}(I)=|\{u\mbox{ monomials in }G(I)\cap S_d \ : \ m(u)=i\}|\\
m_{i}(I)=|\{u\mbox{ monomials in }G(I) \ : \ m(u)=i\}|.\nonumber
\end{eqnarray}
By the Eliahou-Kervaire formula \cite{EK} (see also \cite[Corollary 7.2.3]{HH2}) we have:
\begin{equation}\label{eliker}
\beta_{i,i+d}(I)=\sum_{k=i}^{n}\binom{k-1}{i}m_{k,d}(I).
\end{equation}
It is convenient to introduce the analog of the $m_{i,d}$'s for all $\ZZ$-graded finitely generated $S$-module $M$: To this aim, for all $i\in\{1,\ldots ,n+1\}$ and $d\in\ZZ$ we set:
\begin{equation}\label{defmidgen}
m_{i,d}(M)=\sum_{k=0}^n(-1)^{k-i+1}\binom{k}{i-1}\beta_{k,k+d}(M).
\end{equation}
The following lemma shows that knowing the $m_{i,d}(M)$'s is equivalent to knowing the Betti table of $M$, and that the two definitions agree when $M=I$ is a stable ideal.
\begin{lemma}\label{m_ib_i}
Let $M$ be a $\ZZ$-graded finitely generated $S$-module. Then:
\begin{equation}\label{elikergen}
\beta_{i,i+d}(M)=\sum_{k=i}^{n+1}\binom{k-1}{i}m_{k,d}(M).
\end{equation}
\end{lemma}
\begin{proof}
Set $m_{k,d}=m_{k,d}(M)$ and $\beta_{i,j}=\beta_{i,j}(M)$. By the definition of the $m_{k,d}$'s we have the following identity in $\ZZ[t]$:
\[\sum_{k=1}^{n+1} m_{k,d}t^{k-1}=\sum_{i=0}^n\beta_{i,i+d}(t-1)^i.\]
Replacing $t$ by $s+1$, we get the identity of $\ZZ[s]$
\[\sum_{k=1}^{n+1}m_{k,d}(s+1)^{k-1}=\sum_{i=0}^n\beta_{i,i+d}s^i,\]
that implies the lemma.
\end{proof}
Let us define also the coarser invariants:
\begin{equation}\label{defmigen}
m_i(M)=\sum_{d\in\ZZ}m_{i,d}(M) \ \ \ \forall \ i=1,\ldots ,n+1.
\end{equation}
%Let us recall the Eliahou-Kervaire formula for a stable monomial ideal $I\subset P$:
%\begin{equation}\label{eliker}
%\beta_i(I)=\sum_{k=1}^nm_k(I)\binom{k-1}{i}
%\end{equation}
%The analogous formula for the graded Betti numbers is:
%\begin{equation}\label{elikergraded}
%\beta_{i,i+d}(I)=\sum_{k=1}^nm_{k,d}(I)\binom{k-1}{i}
%\end{equation}
%From Lemma \ref{m_ib_i} and \eqref{defmid} follows that a stable ideal generated in degree $d$ has a $d$-linear resolution. Furthermore, if $I$ is a stable ideal, then $I_{\langle d\rangle}$ is stable for all natural numbers $d$. So any stable ideal is componentwise linear.
%
%When $M=I$ is a stable monomial ideal we will consider \eqref{defmid} the definition of the $m_{i,d}$'s, and we will refer to \eqref{eliker} as the Eliahou-Kervaire formula.

%%%%%%%%%%%%%%%%%%%%%%%%%%%%%%%%%%%%%%%%%%%%%%%%%%%%%%%%%%%%%%%%%
%%%%%%%%%%%%%%%%%%%%%%%%%%%%%%%%%%%%%%%%%%%%%%%%%%%%%%%%%%%%%%%%%%

\section{Graded Betti numbers of componentwise linear ideals}\label{seccomp}

In this section we want to discuss the problem of characterizing the graded Betti numbers of a componentwise linear ideal $I\subset S=K[x_1,\ldots,x_n]$. This is a difficult task, in fact we are not going to solve the problem, rather we are going to explain why it is problematic. Such an issue is equivalent to characterize the possible graded Betti numbers of a strongly stable monomial ideal of $S$. In fact, in characteristic $0$ this is true because the generic initial ideal of any ideal $I$ is strongly stable \cite[Theorem 15.23]{Ei}. Moreover, if $I$ is componentwise linear and the term order is degree reverse lexicographic, then the graded Betti numbers of $I$ are the same of those of $\Gin(I)$ by a result of Aramova, Herzog and Hibi in \cite{AHH}. In positive characteristic it is still true that for a degree reverse lexicographic order the graded Betti numbers of $I$ are the same of those of $\Gin(I)$, provided that $I$ is componentwise linear. But in this case $\Gin(I)$ might be not strongly stable. However, it is known that, if we start from a componentwise linear ideal, it is stable (see Conca, Herzog and Hibi \cite[Lemma 1.4]{CHH}). The graded Betti numbers of a stable ideal do not depend from the characteristic, because of the Eliahou-Kervaire formula \eqref{eliker}. So to compute the graded Betti numbers of $\Gin(I)$ we can consider it in characteristic $0$. Let us call $J$ the ideal $\Gin(I)$ viewed in characteristic $0$. The ideal $J$, being stable, is componentwise linear, so we are done by what said above. Summarizing, we showed:

\begin{prop}\label{bc=bs}
The following sets coincide:
\begin{enumerate}
\item[{\em (1)}] $\{\mbox{Betti tables $(\beta_{i,j}(I))$ where $I\subset S$ is componentwise linear}\}$;
\item[{\em (2)}] $\{\mbox{Betti tables $(\beta_{i,j}(I))$ where $I\subset S$ is strongly stable}\}$;
\end{enumerate}
\end{prop}

Before beginning the discussion on graded Betti numbers, it is worthwhile to notice that to characterize the total Betti numbers of a componentwise linear ideal is an easy task. Indeed, together with Proposition \ref{bc=bs}, the following remark of Murai yields the answer:

\begin{remark}\label{murai}(Murai).
Let $(m_1,\ldots ,m_n)$ be a sequence of natural numbers. The following are equivalent:
\begin{itemize}
\item[(i)] $m_1=1$ and $m_{i+1}=0$ whenever $m_i=0$.
\item[(ii)] There exists a strongly stable ideal $I\subset S$ such that $m_i(I)=m_i$ for any $i=1,\ldots ,n$.
\end{itemize}
That (ii) $\implies$ (i) is very easy to show. For the reverse implication, given a sequence $(m_1,\ldots ,m_n)$ satisfying (i), set $k=\max_\ell\{m_\ell\neq 0\}$. By assumption we have $m_i\geq 1$ for all $i=1,\ldots ,k$, therefore it makes sense to define the following monomial spaces for each $j=1,\ldots ,k-1$:
\[V_j=\bigg\langle \prod_{i=1}^{j-2} x_i^{m_{i+1}-1}\cdot x_{j-1}^{m_j-1}x_j^{m_{j+1}}, \ \ \ \prod_{i=1}^{j-2} x_i^{m_{i+1}-1}\cdot x_{j-1}^{m_j-2}x_j^{m_{j+1}+1}, \ \ \ \ldots , \ \ \ \prod_{i=1}^{j-2} x_i^{m_{i+1}-1}\cdot x_j^{m_j+m_{j+1}-1}\bigg\rangle.\]
We also define:
\[V_k=\bigg\langle\prod_{i=1}^{k-2} x_i^{m_{i+1}-1}\cdot x_{k-1}^{m_k-1}x_k, \ \ \ \prod_{i=1}^{k-2} x_i^{m_{i+1}-1}\cdot x_{k-1}^{m_k-2}x_k^2, \ \ \ \ldots , \ \ \ \prod_{i=1}^{k_2} x_i^{m_{i+1}-1}\cdot x_k^{m_k}\bigg\rangle.\]
Clearly, for all $j=1,\ldots ,k$, we have $w_i(V_j)=m_j$ if $i=j$ and $w_i(V_j)=0$ otherwise. Set:
\[I=\bigg(\bigoplus_{j=1}^kV_k\bigg)\subset S.\]
It is easy to see that $I$ is a strongly stable monomial ideal and that $\langle G(I)\rangle =\bigoplus_{j=1}^kV_k$, so we get (ii).
\end{remark}

Actually, also the possible Betti numbers of an ideal with linear resolution are known. Proposition \ref{bc=bs} and \cite[Proposition 3.8]{Mu} yield the following:

\begin{prop}\label{mainlinear}
Let $m_1,\ldots ,m_n$ be a sequence of natural numbers. Then the following are equivalent:
\begin{enumerate}
\item[{\em (1)}] There exists a homogeneous ideal $I\subset S$ with $d$-linear resolution such that $m_k(I)=m_k$ for all $k=1,\ldots ,n$;
\item[{\em (2)}] $(m_1,\ldots ,m_n)$ is an $O$-sequence such that $m_2\leq d$, that is:
\begin{enumerate}
\item[{\em (a)}] $m_1=1$;
\item[{\em (b)}] $m_2\leq d$;
\item[{\em (c)}] $m_{i+1}\leq m_{i}^{\langle i-1 \rangle}$ for any $i=2,\ldots ,n-1$.
\end{enumerate}
\end{enumerate}
\end{prop}

\begin{discussion}\label{remop}
In \cite[Proposition 3.8]{Mu} is shown the equivalence between point (2) of Proposition \ref{mainlinear} and the existence of a strongly stable ideal $I\subset S$ generated in degree $d$ such that $m_k(I)=m_k$ for all $k=1,\ldots ,n$. Looking carefully at the proof, one can see that, given an $O$-sequence $(m_1,\ldots ,m_n)$, a special strongly stable ideal $I\subset S$ such that $m_k(I)=m_k$ shows up. This special strongly stable ideal, which we are going to define more explicitly, will play a crucial role throughout the paper. 

Given an $O$-sequence $(m_1,\ldots ,m_n)$ such that $m_2\leq d$, set:
\[ V_i = \{\mbox{biggest $m_i$ monomials $u\in S_d$ such that $m(u)=i$}\}\]
for all $i=1,\ldots ,n$. Then the discussed ideal is:
\[I=(\bigcup_{i=1}^n V_i)\subset S.\]
We will refer to $I$ as the {\it piecewise lexsegment} monomial ideal (of type $(d,(m_1,\ldots ,m_n))$). From what said above, it follows that
the piecewise lexsegment of type $(d,(m_1,\ldots ,m_n))$ is strongly stable if and only if $(m_1,\ldots ,m_n)$ is a $O$-sequence such that $m_2\leq d$.
\end{discussion}

Let $I\subset S$ be a strongly stable monomial ideal. Notice that both $I_{\langle j\rangle}$ and $\m I$, where $j$ is a natural number and $\m = (x_1,\ldots, x_n)$ is the graded maximal ideal of $S$, are strongly stable.  For all $j\in \NN$ and $i=1,\ldots ,n$, we define:
\[\mu_{i,j}(I)=m_i(I_{\langle j\rangle}).\]
To know the matrix $(m_{i,j}(I))$ or $(\mu_{i,j}(I))$ are equivalent issues: Indeed, if $J\subset S$ is a strongly stable monomial ideal, then for all $i=1,\ldots ,n$:
\[m_i(\m J)=\sum_{q=1}^im_q(J).\]
Therefore we have the formula:
\begin{equation}\label{mu&m}
m_{i,j}(I)=m_i(I_{\langle j\rangle})-m_i(\m I_{\langle j-1 \rangle})=\mu_{i,j}(I)-\sum_{q=1}^i \mu_{q,j-1}(I)
\end{equation}
that implies that we can pass from the $\mu_{i,j}$'s to the $m_{i,j}$'s. It also follows from this formula that we can do the converse path by induction on $j$, because $\mu_{i,d}(I)=m_{i,d}(I)$ if $d$ is the smallest degree in which $I$ is not zero. Therefore, using Proposition \ref{bc=bs}, to characterize the possible Betti tables of the componentwise linear ideals is equivalent to answer the following question:

\begin{question}
What are the possible matrices $\M(I)=(\mu_{i,j}(I))$ where $I\subset S$ is a strongly stable ideal?
\end{question}

We will refer to $\M=\M(I)$ as the {\it matrix of generators} of the strongly stable ideal $I$. We will feature $\M$ as follows:
\[\mathcal{M}=\left(
                   \begin{array}{cccccc}
                     \mu_{1,1} & \mu_{2,1} & \mu_{3,1} & \cdots & \cdots & \mu_{n,1} \\
                     \mu_{1,2} & \mu_{2,2} & \mu_{3,2} & \cdots & \cdots & \mu_{n,2} \\
                     \mu_{1,3} & \mu_{2,3} & \mu_{3,3} & \cdots & \cdots & \mu_{n,3} \\
                     \vdots & \vdots & \vdots & \cdots & \cdots & \vdots \\
                   \end{array}
                 \right)\]
We can immediately state the following:

\begin{prop}\label{main}
Let $\mathcal{M}=(\mu_{i,j})$ be the matrix of generators of a strongly stable monomial ideal $I\subset S$. Then the following conditions hold:
\begin{itemize}
\item[{\em (i)}] Each non-zero row vector $(\mu_{1,j},\mu_{2,j},\ldots ,\mu_{n,j})$ of $\mathcal{M}$ is an $O$-sequence such that $\mu_{2,j}\leq j$.
\item[{\em (ii)}] For all $i$ and $j$ one has $\mu_{i,j}\geq \sum_{q = 1}^i\mu_{q,j-1}$.
\end{itemize}
\end{prop}
\begin{proof}
Condition (i) follows from Proposition \ref{mainlinear} since $I_{\langle j\rangle}$ has a $j$-linear resolution for all $j$ greater than or equal to the lower degree in which $I$ is not zero. Condition  (ii) follows from \eqref{mu&m}.
\end{proof}

Notice that the Noetherianity of $S$ (or if you prefer conditions (i) and (ii) of Proposition \ref{main}) implies that there exists $m\in\NN$ such that $\mu_{i,j}(I) = \sum_{q = 1}^i\mu_{q,j-1}(I)$ for all $j>m$ and $i\in\{1,\ldots ,n\}$. So, though $\M$ has infinitely many rows, the relevant ones are just a finite number, and in the examples we will write just them.

One may expect that the conditions described in Proposition \ref{main} are sufficient. But this is not the case at all:

\begin{example}\label{obstructionscomplin}
One obstruction is illustrated already by Remark \ref{murai}: Consider the matrix
\[\mathcal{M}=\left(
                   \begin{array}{cccc}
                     0 & 0 & 0 & 0 \\
                     \vdots & \vdots & \vdots & \vdots \\
                     0 & 0 & 0 & 0 \\
                     1 & d & 0 & 0 \\
                     1 & d+1 & d+1 & k
                   \end{array}
                 \right)\]
where the first nonzero row from the top is the $d$th and $d+1<k\leq (d+1)^{\langle 2\rangle}$. Such a matrix clearly satisfies the necessary conditions of Proposition \ref{main}. However, if there existed a strongly stable ideal $I\subset K[x_1,x_2,x_3,x_4]$ with matrix of generators $\M$, then it would satisfy $m_1(I)=1$, $m_2(I)=d$, $m_3(I)=0$ and $m_4(I)=k-d-1>0$, a contradiction to Remark \ref{murai}. The first matrix of this kind is:
\[\mathcal{M}=\left(
                   \begin{array}{cccc}

                     0 & 0 & 0 & 0 \\
                     1 & 2 & 0 & 0 \\
                     1 & 3 & 3 & 4
                   \end{array}
                 \right).\]
The explained obstruction gives rise to a class of counterexamples. However, such a class does not fill the gap between the existence of a strongly stable ideal with matrix of generators $\M$ and the necessary conditions of Proposition \ref{main}. Let us look at the following matrix.
\[\mathcal{M}=\left(
                   \begin{array}{cccc}
                     0 & 0 & 0 & 0 \\
                     0 & 0 & 0 & 0 \\
                     0 & 0 & 0 & 0 \\
                     0 & 0 & 0 & 0 \\
                     1 & 3 & 2 & 2 \\
                     1 & 4 & 6 & 9
                   \end{array}
                 \right).\]
One can check that the necessary conditions described in Proposition \ref{main} hold. However one can show that there is no strongly stable monomial ideal $I\subset K[x_1,\ldots ,x_4]$ with $\mathcal{M}$ as matrix of generators. Notice that such an ideal would have $m_1(I)=1$, $m_2(I)=3$, $m_3(I)=2$ and $m_4(I)=3$, which does not contradict Remark \ref{murai}.
\end{example}

\begin{example}\label{nagelroemer}
Obviously the property of having linear resolution can be detected looking at the graded Betti numbers. In the following example we show that this is not anymore true for componentwise linear ideals, and this strengthens the impression that to give a complete characterization of the possible graded Betti  numbers of a componentwise linear ideal is probably a hard task. More precisely, we are going to exhibit two ideals $I$ and $J$, one componentwise linear and one not, with the same Betti tables. This answers negatively a question raised in \cite[Question 1.1]{NR}.

Consider the ideals of $K[x_1,x_2,x_3]$:
\[I = (x_1^4,  \    x_1^3x_2 , \    x_1^2x_2^2, \    x_1x_2^3,  \   x_2^4,  \   x_1^3x_3, \    x_1^2x_2x_3^2,  \   x_1^2x_3^3, \    x_1x_2^2x_3^2)\]
and
\[J = (x_1^4,   \  x_1^3x_2 ,    \ x_1^2x_2^2,   \  x_1^3x_3,   \  x_1x_2^2x_3,   \  x_1x_2x_3^2,   \  x_1x_2^4,   \  x_1^2x_3^3, \    x_2^4x_3).\]
Notice that $I$ and $J$ are generated in degrees $4$ and $5$. By CoCoA \cite{cocoa} one can check that $I$ and $J$ have the same Betti table, namely:
\[\left(
                   \begin{array}{cccc}
                     0 & 0 & 0 \\
                     0 & 0 & 0 \\
                     0 & 0 & 0 \\
                     6 & 6 & 1 \\
                     3 & 6 & 3
                   \end{array}
                 \right).\]
One can easily check that $I$ is strongly stable, so in particular it is componentwise linear. On the contrary $J$ is not componentwise linear, since $J_{\langle 4\rangle}=(x_1^4,     x_1^3x_2 ,     x_1^2x_2^2,     x_1^3x_3,     x_1x_2^2x_3,     x_1x_2x_3^2)$, as one can check by CoCoA,  has not a $4$-linear resolution.
\end{example}

\vskip.3mm

We are going  to explain some reasons why the conditions of Proposition \ref{main} for a matrix $\mathcal{M}$ are in general not sufficient to have a strongly stable ideal corresponding to it. By the discussion after Proposition~\ref{mainlinear}, we know that  for a given
sequence $(m_1,\ldots ,m_n)$ of integers,  there exists a strongly stable monomial ideal $J$ generated
in degree $d$ such that $m_i(J) = m_i$ if and only if the piecewise lexsegment ideal of type $(d,(m_1,\ldots ,m_n))$ is strongly stable. Unfortunately, even
if $J$ is a piecewise lexsegment, the ideal  $\m J$ is not necessarily a piecewise lexsegment. For instance, keeping in mind the last matrix in Example \ref{obstructionscomplin}, the piecewise
lexsegment ideal of type $(5,(1,3,2,2))$ is:
\[J = (x_1^5, x_1^4x_2, x_1^3x_2^2,x_1^2x_2^3, x_1^4x_3, x_1^3x_2x_3, x_1^4x_4, x_1^3x_2x_4).\]
However $u=x_1^3x_3^3\notin \m J$, whereas $v= x_1^2 x_2^3x_3\in \m J$. Since $v$ is lexicographically
smaller than $u$ and $m(u)=m(v)=3$, $\m J$ is not a piecewise lexsegment ideal.
This fact does not make any troubles when the number of variables is at most three, as we can see later. To see this, we need the following more general proposition.

%Let $J\subset K[x_1,\cdots ,n]$ be an ideal. We denote by $J_1$  the elimination ideal $J\cap K[x_1,\cdots, x_{n-1}]$. It is clear that if $J$ is strongly stable ideal (resp. piecewise-lexsegment
%generated in degree $d$) so is $J_1$.

\begin{prop}\label{suffcomplin}
Let $\mathcal{M}=(\mu_{i,j})$ be a matrix. Then $\mathcal{M}$ is the matrix  of generators of a strongly stable monomial ideal $I\subset S$, provided that the following conditions hold:

\begin{enumerate}
\item[{\em (1)}]
Each nonzero row vector $(\mu_{1,j},\ldots ,\mu_{n,j})$ of $\mathcal{M}$ is an $O$-sequence with $\mu_{2,j}\leq j$.
\item[{\em (2)}]
For all  $i\in\{1,\ldots ,n\}$ one has $\mu_{i,j}\geq \sum_{q=1}^i\mu_{q,j-1}$.
\item[{\em (3)}]
If $m\in\NN$ is the least number such that $\mu_{i,h}= \sum_{q=1}^i\mu_{q,h-1}$ for each $h>m$, then for each $j\leq m$ there exists a strongly stable ideal $B(j)\subset S$ generated in degree $j$ such that $m_i(B(j))=\mu_{i,j}$ for all $i=1,\ldots ,n$ and $B(j)\cap K[x_1,\ldots ,x_{n-1}]_{j+1}$ is a piecewise lexsegment monomial space.
\end{enumerate}
\end{prop}
\begin{proof}
Let $d$ be the least natural number such that the row vector $(\mu_{1,d},\ldots ,\mu_{n,d})$ is nonzero. We claim to have built a strongly stable ideal $I(j)\subset S$ such that $\mu_{i,k}(I(j))=\mu_{i,k}$ for any $k\leq j$ and $I(j)_{\langle j\rangle}\cap K[x_1,\ldots ,x_{n-1}]=B(j)\cap K[x_1,\ldots ,x_{n-1}]$. If $j = m$, then the desired ideal is $I=I(m)$. If not, however we can assume $j\geq d$ ($I(d)=B(d)$). We set \[m_{i,j+1}=\mu_{i,j+1}- \sum_{q=1}^i\mu_{q,j}.\]
Let $L(j+1)$ be the ideal generated by  the biggest $m_{i,j+1}$ ($i=1,\ldots ,n$) monomials $u\in S_{j+1}\setminus I(j)$ such that $m(u) = i$ (they exist thanks to condition (1)). Set $I(j+1)=I(j)+L(j+1)$. Clearly the first $j+1$ rows of the matrix of generators of $I(j+1)$ coincide with the ones of $\mathcal{M}$. Then notice that $I(j+1)_{\langle j+1\rangle}\cap K[x_1,\ldots ,x_{n-1}]$ is the piecewise lexsegment of type $(j+1,(\mu_{1,j+1},\cdots,\mu_{n-1,j+1}))$ by construction. Because $B(j+1)\cap K[x_1,\ldots ,x_{n-1}]_{j+2}$ is a piecewise lexsegment monomial space, $B(j+1)\cap K[x_1,\ldots ,x_{n-1}]$ is forced to be the piecewise lexsegment of type $(j+1,(\mu_{1,j+1},\cdots,\mu_{n-1,j+1}))$. So we get the equality:
\[I(j+1)_{\langle j+1\rangle}\cap K[x_1,\ldots ,x_{n-1}]=B(j+1)\cap K[x_1,\ldots ,x_{n-1}].\]
To conclude the proof, we have to show that $I(j+1)$ is strongly stable. This reduces to show that, if $u\in I(j+1)$ is a monomial of degree $j+1$ with $m(u)=n$, then $(u/x_i)x_k$ belongs to $I(j+1)$ for all $1\leq k<i\leq n$ such that $x_i\mid u$. We consider two cases. If $u\in I(j)$ we are done,  because $I(j)\subset I(j+1)$ is strongly stable. If $u\in L(j+1)$, then we consider the monomial ideal
\[T(j+1)=(v\in G(I(j+1)_{\langle j+1\rangle}):\  \ m(v)<n \ {\rm or}\ v\geq u)\subset I(j+1).\]
Observe that $T(j+1)$ is a piecewise lexsegment ideal of type $(j+1,(\mu_{1,j+1},\ldots,\mu_{n-1,j+1}, a))$, where $a\leq \mu_{n,j+1}$. Since $(\mu_{1,j+1},\ldots,\mu_{n-1,j+1}, a)$ is an $O$-sequence with $\mu_{2,j+1}\leq j+1$, it follows that  $T(j+1)$ is strongly stable by the discussion after Proposition \ref{mainlinear}. Thus for each $1\leq k<i\leq n$ such that $x_i|u$, we have that  $(u/x_i)x_k\in T(j+1)\subset I(j+1)$.
%
%
%
%
%
%For all $\ell\geq 1$, set $J(\ell)$ the piecewiselexsegment of type
%$(\ell,(\mu_{1,\ell},\cdots,\mu_{n,\ell}))$. For each  $i\in\{1,\ldots ,n\}$ set
%\[m_{i,\ell}=\mu_{i,\ell}- \sum_{q=1}^i\mu_{q,\ell -1}.\]
%Let $L(\ell)$ be the ideal generated by  the biggest $m_{i,\ell}$ ($i=1,\ldots ,n$) monomials $u\in P_j\setminus J(\ell-1)$ such that $m(u) = i$. Set $I(\ell)=J(\ell-1)+L(\ell)$. It is straightforward to check that the first $\ell$ rows of matrix of generators of $I(\ell)$ coincide with the ones of $\mathcal{M}$. So it remains to be  shown that $I(\ell)$ is strongly stable.
%
%First of all notice that since the condition (2) holds for $J_1$, the ideal $I_1=I\cap K[x_1,\ldots ,x_{n-1}]$ is a strongly stable ideal.
%So we have to show that, if $u\in I$ is a monomial of degree $d$ with $m(u)=n$, then $(u/x_i)x_j$ belongs to $I$ for all $1\leq j<i\leq n$ such that $x_i\mid u$. We consider two cases. If $u\in J$,  again we are done,  because $J\subset I$ is strongly stable. If $u\in I\setminus J$, then we consider the monomial ideal
%\[T=(v\in G(I_{\langle d\rangle}):\  \ m(v)<n \ {\rm or}\ v\geq u)\subset I.\]
%Observe that $T$ is a piecewise lexsegment ideal of type $(d,(\mu_{1,d},\ldots,\mu_{n-1,d}, a))$, where $a\leq \mu_{n,d}$. Since $(\mu_{1,d},\ldots,\mu_{n-1,d}, a)$ is an $O$-sequence with $\mu_{2,d}\leq d$, it follows that  $T$ is a strongly stable ideal. Thus for each $1\leq j<i\leq n$ such that $x_i|u$, we have that  $(u/x_i)x_j\in T\subset J$.
\end{proof}

\begin{corollary}\label{3variables}
Let $\mathcal{M}=(\mu_{i,j})$ be a  matrix with $3$ columns. Then $\mathcal{M}$ is the matrix  of generators of a strongly stable monomial ideal $I\subset K[x_1,x_2,x_3 ]$ if and only if  the following conditions hold:
\begin{enumerate}
\item[{\em (1)}]  Each non-zero row vector $(\mu_{1,j},\mu_{2,j},\mu_{3,j})$ of $\mathcal{M}$ is an $O$-sequence with $\mu_{2,j}\leq j$.
\item[{\em (2)}]  For all $j\in \NN$ one has $\mu_{2,j}\geq \mu_{1,j-1}+\mu_{2,j-1}$ and $\mu_{3,j}\geq \mu_{1,j-1}+\mu_{2,j-1}+\mu_{3,j-1}$.
\end{enumerate}
\end{corollary}

\begin{proof}
The conditions are necessary from Proposition \ref{main}. Furthermore, since an ideal $I\subset K[x_1,x_2]$ generated in one degree is piecewise lexsegment if and only if it is strongly stable, we automatically have condition (3) of Proposition \ref{suffcomplin}.
%First note that a monomial ideal $I\subset K[x_1,x_2]$ is a piecewise lexsegment ideal generated in a fixed degree if and only if $I$ is a lexsegment ideal generated in that degree. It follows that if $I\subset K[x_1,x_2]$ is a piecewise lexsegment so is $\n I$ where $\n=(x_1,x_2)\subset K[x_1,x_2]$.
%
%In order to prove the result it is enough to find a strongly stable monomial ideal corresponding to a matrix $\mathcal{M}=(\mu_{i,j})_{3\times d}$
%which satisfies in conditions (1) and (2).
%
%Set $J_0=(0)$ and $m_{i,j}=\mu_{i,j}-\sum_{l=1}^i\mu_{l,j-1}$, and let $L_j$ be the monomial ideal generated by  the first $m_{i,j}$ ($i=1,2,3$)  monomials $u\in S_j\setminus J_{j-1}$, such that $m(u) = i$ and $J_j=\m J_{j-1}+L_j$. We claim that $J=\sum J_j$ is the desired ideal.
%
% The fact that $\mathcal{M}$ is the matrix of generators of $J$ is an immediate consequence of the construction of  $J$. Moreover for each $j$, $J_{j-1}+J_j$ is the ideal obtained in the proof of previous lemma. So $J$ is strongly stable.
\end{proof}

\vskip .3mm

Although the complete characterization of the matrix of generators of an arbitrary strongly stable ideal seems to be very complicated, based on the fact that the lexsegment property of an ideal is preserved under multiplication by the maximal ideal $\m$,  one may expect a characterization for the matrix of generators of lexsegment ideals.  For answering this question, first we define the concept of a $d$-lex sequence.

\begin{definition}
A  sequence of non-negative integers $m_1,\cdots, m_n$ is called a $d$-lex sequence, if there exists a lexsegment ideal $L\subset S$ generated in degree $d$ such that $m_i(L)=m_i$ for all $i$.
\end{definition}

Because if $I\subset S$ is a lexsegment ideal, then $\m I$ is still a lexsegment ideal, we clearly have that $\mathcal{M}=(\mu_{i,j})$ is the matrix  of generators of a lexsegment ideal
 if and only if the following conditions hold:
\begin{enumerate}
\item[(1)] Each non-zero row vector $(\mu_{1,j},\mu_{2,j},\ldots ,\mu_{n,j})$ of $\mathcal{M}$ is a $j$-lex sequence.
\item[(2)]  For all $i$ and $j$ one has $\mu_{i,j}\geq \sum_{q=1}^i\mu_{q,j-1}$.
\end{enumerate}
Therefore to characterize the matrix of generators of lexsegment ideals we need to characterize arbitrary $d$-lex sequences. To do this, we have to recall the definition of  the natural decomposition of the complement set of monomials belonging to a lexsegment ideal generated in a fixed  degree. In what follows we denote by $[x_t,\ldots , x_n]_r$ ($1\leq t\leq n$) the set of all monomials of degree $r$ in the variables $x_t,\ldots , x_n$.

\begin{definition}
Let $u=x_{j(1)}\ldots x_{j(d)}\in S_d$ ($1\leq j(1)\leq \cdots \leq j(d)\leq n$) be a monomial and set $L_{<u}=\{ v\in S_d\ | \ v< u\}$. Following the method described in \cite[page 159]{BH} (where $L_{<u}$ is denoted by $\mathcal{L}_u$) we can partition the set $L_{<u}$ as:
\[L_{<u}=\bigcup_{i=1}^d[x_{j(i)+1},\ldots , x_n]_{d-i+1}\cdot x_{j(1)}\cdots x_{j(i-1)},\]
which is called the {\it natural decomposition} of $L_{<u}$.
\end{definition}

Before proving the next result, notice that the powers of the maximal ideal are lexsegment ideals, and the following formula holds for their $d$-lex sequences:
\begin{equation}\label{powerm}
m_i(\m^d)=\displaystyle {i+d-2\choose d-1}.
\end{equation}
% \begin{lemma}
% Let $S=k[x_1,\ldots,x_n]$ and $\m=(x_1,\ldots ,x_n)$ be the maximal ideal. then for each positive integer $k$,  $m_i(\m^k)={i+k-2\choose k-1}$.
% \end{lemma}
%
%\begin{proof}
%It is enough to notice that  $G(\m^k)_i=\{u\in G(\m^k)\ : \ m(u)=i\}=\{x_iv\ : v\in G(x_1,\ldots ,x_i)^{k-1}\}$.
%Thus $m_i(\m^k)=\dim_k((K[x_1,\ldots ,x_i])_{k-1})={i+k-2\choose k-1}$.
%\end{proof}

\begin{prop}\label{d-lex}
Let $m_1,\ldots, m_n$ be a sequence of natural numbers and let $\mu=\sum_{i=1}^n m_i$. Suppose that
$$\ell={n+d-1\choose d}-\mu=\sum_{j=1}^d{k(j)\choose j}$$
is the $d$-th Macaulay representation of $\ell$. Then $m_1,\cdots, m_n$ is a $d$-lex sequence, if and only if
$$m_i={i+d-2\choose d-1}-\sum_{j=1}^d{k(j)-n+i-1\choose j-1}.$$
\end{prop}
\begin{proof}
The sequence $m_1,\ldots, m_n$ is a $d$-lex sequence if and only if $I_u=(L_{\geq u})$ satisfies $m_i(I_u)=m_i$ for all $i=1,\ldots ,n$, where $u$ is the $\mu$th biggest monomial of degree $d$. Let us write $u=x_{j(1)}\cdots x_{j(d)}, 1\leq j(1)\leq \cdots \leq j(d)\leq n$. By the natural decomposition of $L_{<u}$ we have:
\[\ell=|L_{<u}|=\sum_{p=1}^d \dim_K[x_{j(i)+1},\ldots , x_n]_{d-p+1}=\sum_{p=1}^d{n-j(p)+d-p\choose d-p+1}.\]
Setting $t=d-i+1$ and $k(t)=n-j(d-t+1)+t-1$, we have that $\sum_{t=1}^d{k(t)\choose t}$ is the $d$th Macaulay representation of $\ell$. The natural decomposition of $L_{<u}$ and \eqref{powerm} show that
\begin{eqnarray*}
m_i((L_{<u}))=
\sum_{t=1}^d m_i(x_{j(d-t+1)+1},\ldots , x_n)^{t}=\sum_{t=1}^d{i-j(d-t+1)+t-2\choose t-1}=\sum_{t=1}^d{k(t)-n+i-1\choose t-1}.
\end{eqnarray*}
Because
\[m_i(I_u)=m_i(\m^d)-m_i((L_{<u})),\]
we get the conclusion thanks to \eqref{powerm}.
%\vskip2mm
%
% But $\mathcal{L}_u=\bigcup_{i=1}^d[x_{j(i)+1},\ldots , x_n]_{d-i+1}x_{j(1)}\cdots x_{j(i-1)},$ thus
%  $$x=|\mathcal{L}_u|=\sum_{i=1}^d \dim_k[x_{j(i)+1},\ldots , x_n]_{d-i+1}=\sum_{i=1}^d{n-j(i)+d-i\choose d-i+1}.$$
%
%
%  To convert the above sum as the $d$-th Macaulay representation of $x$, we
%  let $t=d-i+1$. Therefore,
%  $x=|\mathcal{L}_u|=\sum_{t=1}^d{n-j(d-t+1)+t-1\choose t}$.
%
%\vskip2mm
%
%  Note that if $t<t'$ then $j(d-t'+1)\leq j(d-t+1)$ which implies
%  $n-j(d-t+1)+t-1<n-j(d-t'+1)+t'-1$. So if we let $k(t)=n-j(d-t+1)+t-1$, then $|\mathcal{L}_u|=\sum_{t=1}^d{k(t)\choose t}$ is the $d$-th Macaulay representation of $|\mathcal{L}_u|$. Note that  in this sum ${k(t)\choose t}=\dim_k[x_{j(d-t+1)+1},\ldots , x_n]_{t}$.
%
%  The natural decomposition of $\mathcal{L}_u$ also shows that
%
%
%  \begin{eqnarray*}
%&m_i(\mathcal{L}_u)&=
%\sum_{t=1}^d m_i(x_{j(d-t+1)+1},\ldots , x_n)^{t}=\sum_{t=1}^d{i-j(d-t+1)+t-2\choose t-1}\\ &&=\sum_{t=1}^d{k(t)-n+i-1\choose t-1},
%\end{eqnarray*}
% and from it we get the desired formula.
%
\end{proof}

%\begin{thm}\label{main2}
% A matrix $\mathcal{M}=(\mu_{i,j})$ is a matrix  of generators of a lexsegment ideal $I\subset P$
% if and only if the following conditions hold:
%\begin{enumerate}
%\item[{\em (1)}] Each non-zero column vector $(\mu_{1,j},\mu_{2,j},\ldots ,\mu_{n,j})$ of $\mathcal{M}$ is a $j$-lex sequence.
%\item[{\em (2)}]  For all $i$ and $j$ one has $\mu_{i,j}\geq \sum_{q=1}^i\mu_{q,j-1}$.
%\end{enumerate}
%\end{thm}

We recall that a homogeneous ideal $I\subset S$ is said to be {\it Gotzmann} if the number of minimal generators of $\mm I_{\langle j\rangle}$ is the smallest possible for every $j\in\NN$, namely is equal to:
\[\binom{n+j}{j+1}-\left(\binom{n+j-1}{j}-\mu_j\right)^{\langle j\rangle},\]
where $\mu_j$ is the number of minimal generators of $I_{\langle j\rangle}$. The graded Betti numbers of a Gotzmann ideal coincide with its associated lexsegment ideal, see \cite{HH}. Therefore Proposition \ref{d-lex} characterizes also the graded Betti numbers of Gotzmann ideals.
%\vskip3mm
%
%We remark that applying  Corollary \ref{betti-generator} and Theorem \ref{d-lex},  we can read the Betti numbers of a lexsegment ideal $L\subset K[x_1,\cdots , x_n]$ from its Hilbert function. Among them, $\beta_{n-1,j}$ can be found easily. To see it we introduce the following notation.
%
%\vskip3mm
%
%Given the integers $a,d$, let $a=\sum_{t=1}^d{k(t)\choose t}$ be the $d-$th Macaulay representation of $a$. We define
%$a_{[d]}=\sum_{t=1}^d{k(t)-1\choose t-1}$.
%
%\vskip3mm
%
%
%\begin{corollary}
%Let $H:\NN\to \NN$ be a Hilbert function and $L$ be the corresponding lexsegment ideal. Let $n=H(1)$ then $$\beta_{n-1,n+i}(L)=H_{i}-({H_{i+1}})_{[ i+1]}\ \ \ {\text{for each}}\  i\geq 1.$$
%\end{corollary}
%
%\begin{proof}
%It is clear that for each $i\geq 1$, $|G(L_{\langle i\rangle})|={n+i-1\choose i}-H(i)$.  Let $$H(i+1)=\sum_{t=1}^{i+1}{k(t)\choose t}$$ be the $(i+1)-$th Macaulay representation of $H(i+1)$ for each $i\geq 0$.
%
%By Corollary \ref{betti-generator},
%$$\beta_{n-1,n+i}(L)=\mu_{n,i+1}(L)-\sum_{l=1}^n\mu_{l,i}(L)$$
%where $$\sum_{l=1}^n\mu_{l,i}(L)=|G(L_{\langle i\rangle})|={n+i-1\choose i}-H(i)$$ and by Theorem \ref{d-lex},
%$$\mu_{n,i+1}={n+i-1\choose i}-\sum_{t=1}^{i+1}{k(t)-1\choose t-1}={n+i-1\choose i}-({H_{i+1}})_{[ i+1]}.$$
%
%
%substituting the last two formula in the first one we get the desired formula.
%\end{proof}
%

\section{The possible extremal Betti numbers of a graded ideal}\label{secextremal}

For a fixed $\ell\in\{1,\ldots ,n\}$, $d\in \NN$ and $k\leq \binom{\ell+d-2}{\ell-1}$, we denote by $u(\ell,k,d)$ the $k$th biggest monomial $u\in S_d$ such that $m(u)=\ell$. Or, equivalently, $x_{\ell}$ times the $k$th biggest monomial in $K[x_1,\ldots ,x_{\ell}]_{d-1}$. By $U(\ell ,k,d)$ we denote the ideal of $S$ generated by the set $L_{\geq u(\ell , k, d)}\cap K[x_1,\ldots ,x_{\ell}]$. Notice that $U(\ell ,k,d)$ is not a lexsegment in $S$. However, it is the extension of a lexsegment in $K[x_1,\ldots ,x_{\ell}]$. Furthermore, $U(\ell ,k,d)$ is obviously a piecewise lexsegment in $S$. In this section we need to introduce the following definition: A monomial ideal $I\subset S$ generated in one degree is called {\it piecewise lexsegment up to $\ell$} if $I\cap K[x_1,\ldots ,x_{\ell}]\subset K[x_1,\ldots ,x_{\ell}]$ is piecewise lexsegment.

\begin{remark}\label{keyextremal1}
Notice that, for all $q\in\NN$, denoting by $\mm\subset S$ the maximal irrelevant ideal, $\mm^qU(\ell ,k,d)\cap K[x_1,\ldots,x_{\ell}]$ is equal to $U(\ell ,m_{\ell}(\mm^q U(\ell ,k,d)),d+q))\cap K[x_1,\ldots ,x_{\ell}]$. In particular, $\mm^q U(\ell ,k,d)$ is a piecewise lexsegment up to $\ell$.
\end{remark}

\begin{lemma}\label{keyextremal2}
The ideal $U(\ell ,k,d)\subset S$ is the smallest strongly stable ideal containing the biggest $k$ monomials $u_i\in S_d$ such that $m(u_i)=\ell$ for all $i=1,\ldots ,k$.
\end{lemma}
\begin{proof}
Let $J\subset S$ be the smallest strongly stable ideal containing the biggest $k$ monomials $u_i\in S_d$ such that $m(u_i)=\ell$ for all $i=1,\ldots ,k$. Being the extension of a lexsegment, $U(\ell ,k,d)$ is strongly stable, so that $J\subset U(\ell ,k,d)$. Therefore, let us show the inclusion $U(\ell ,k,d)\subset J$. Let $u$ be a minimal monomial generator of $U(\ell ,k,d)$. So $u$ has degree $d$ and $m(u)\leq \ell$. Actually, we can assume $m(u)<\ell$, otherwise there is nothing to prove. So let us write:
\[u=x_1^{a_1}\ldots x_{\ell-1}^{a_{\ell -1}}.\]
By definition $u> u(\ell,k,d)=x_1^{b_1}\ldots x_{\ell}^{b_{\ell}}$. Set $F=\{i:a_i>b_i\}$. Because $u> u(\ell,k,d)$, we have $F\neq \emptyset$ and $a_j=b_j$ for all $j<i_0=\min\{i:i\in F\}$. If $|F|=1$, then $a_i=b_i$ for all $i_0<i<\ell$ and $b_{\ell}=a_{i_0}-b_{i_0}$, so that $u=x_{i_0}^{a_{i_0}-b_{i_0}}\cdot (u(\ell,k,d)/x_{\ell}^{a_{i_0}-b_{i_0}})\in J$. If $|F|>1$, take $j>i_0$ such that $a_j>b_j$. The monomial $u'=x_{\ell}\cdot (u/x_j)$ is such that $u'>u(\ell ,k,d)$ and $m(u')=\ell$. Therefore $u'\in J$, so that $u=x_j\cdot (u'/x_{\ell})$ belongs to $J$ too.
\end{proof}

The above lemma allows us to characterize the possible extremal Betti numbers of a homogeneous ideal in a polynomial ring. To this aim, we start with a discussion. To $U(\ell ,k,d)$ we can associate the numerical sequence $(m_1,\ldots ,m_{\ell})$ where $m_i=m_i(U(\ell ,k,d))$. Notice that $m_{\ell}=k$. By the discussion done after  Proposition \ref{mainlinear}, if $V$ is a strongly stable monomial ideal generated in degree $d$ such that $m_{\ell}(V)=k$, then there must exist a strongly stable piecewise lexsegment ideal $U$ such that $m_i(U)=m_i(V)$ and containing the $k$ biggest monomials $u\in S_d$ such that $m(u)=\ell$. By Lemma \ref{keyextremal2} $U(\ell,k,d)\subset U$, so that $m_i\leq m_i(V)$ for all $i$.
It is possible to characterize the possible numerical sequences like these. To this purpose, we need to introduce a notion. Given a natural number $a$ and a positive integer $d$, consider the $d$th Macaulay representation of $a$, say $a=\sum_{i=1}^d\binom{k(i)}{i}$. For all integer numbers $j$, we set:
\[a^{\langle d,j\rangle}=\sum_{i=1}^d\binom{k(i)+j}{i+j},\]
where we put $\binom{p}{q}=0$ whenever $p$ or $q$ are negative, and $\binom{0}{0}=1$.
Notice that $a^{\langle d,0\rangle}=a$ and $a^{\langle d,1\rangle}=a^{\langle d\rangle}$.

%Now, given two positive integers $k$ and $\ell$, we define some invariants $v_{k,\ell}^i$ for any $i\in\{1,\ldots ,\ell\}$ as follows:
%\[v_{k,\ell}^i=\min\{a:k\leq a^{\langle i-1,\ell-i\rangle}\} \ \ \ \forall \ i=2,\ldots ,\ell \ \ \mbox{ \ and \ } \ \ v_{k,\ell}^1=1.\]

\begin{lemma}\label{mireferee1}
If $k\leq \binom{\ell+d-2}{\ell-1}$, then:
\[m_i(U(\ell ,k,d))=k^{\langle \ell-1,i-\ell\rangle} \ \ \forall \ i=1,\ldots ,\ell.\]
Furthermore, if $i\geq 2$, then $k^{\langle \ell-1,i-\ell\rangle}=\min\{a:k\leq a^{\langle i-1,\ell-i\rangle}\}$.
%\[m_i(\overline{U(\ell ,k,d)})=v_{k,\ell}^i \ \ \forall \ i=1,\ldots ,\ell.\]
\end{lemma}
\begin{proof}
First we will show that, if $i\geq 2$, then:
\[k^{\langle \ell-1,i-\ell\rangle}=\min\{a:k\leq a^{\langle i-1,\ell-i\rangle}\}.\]
Let us consider the $(\ell-1)$th Macaulay representation of $k$, namely $k=\sum_{j=1}^{\ell -1}\binom{k(j)}{j}$. So 
\[b=k^{\langle \ell-1,i-\ell\rangle}=\sum_{j=\ell-i}^{\ell-1}\binom{k(j)+i-\ell}{j+i-\ell}.\]

If $\max\{j:k(j)<j\}\geq\ell-i$, then the above one is the $(i-1)$th Macaulay representation of $b$: Therefore $b^{\langle i-1,\ell-i\rangle}=k$, so the statement is obvious in this case.

So we can assume that $\max\{j:k(j)<j\}<\ell-i$. In particular, $k(\ell-i)\geq \ell-i$, so that the $(i-1)$th Macaulay representation of $b-1$ is 
\[b-1=\sum_{j=\ell-i+1}^{\ell-1}\binom{k(j)+i-\ell}{j+i-\ell}.\]
Thus $(b-1)^{\langle i-1,\ell-i\rangle}=\sum_{j=\ell-i+1}^{\ell -1}\binom{k(j)}{j}$, which in this case is smaller than $k$. So $b\leq \min\{a:k\leq a^{\langle i-1,\ell-i\rangle}\}$. On the other hand, let us consider the $(i-1)$th Macaulay representation of $b$, namely $b=\sum_{j=1}^{i-1}\binom{h(j)}{j}$. By \cite[Lemma 4.2.7]{BH}, we infer the inequality
\[(h(i-1),\ldots ,h(1))>(k(\ell-1)+i-\ell,\ldots ,k(\ell-i+1)+i-\ell)\]
in the lexicographical order. Of course the inequality keeps to be true when shifting of $\ell-i$, namely
\[(h(i-1)+\ell-i,\ldots ,h(1)+\ell-i)>(k(\ell-1),\ldots ,k(\ell-i+1))\]
in the lexicographical order. Again using \cite[Lemma 4.2.7]{BH}, we deduce that $b^{\langle i-1,\ell-i\rangle}>k$. So $b\geq \min\{a:k\leq a^{\langle i-1,\ell-i\rangle}\}$, that lets us conclude this part.

Let us prove that 
\[m_i(U(\ell ,k,d))=k^{\langle \ell-1,i-\ell\rangle} \ \ \forall \ i=1,\ldots ,\ell.\]
The condition $k\leq \binom{\ell+d-1}{\ell}$ assures that we can construct $V=U(\ell ,k,d)$. The equality is true for $i=1$, because $k^{\langle \ell-1,1-\ell\rangle}=1$. From Proposition \ref{mainlinear} we have, for all $i=2,\ldots ,\ell$:
\[m_{i+1}(V)\leq m_i(V)^{\langle i-1\rangle} \ , \ \ m_{i+2}(V)\leq m_{i+1}(V)^{\langle i\rangle} \ , \ \ \ldots \ , \ \ k=m_{\ell}(V)\leq m_{\ell -1}(V)^{\langle \ell -2\rangle}.\]
Putting together the above inequalities, we get:
\[k\leq m_i(V)^{\langle i-1,\ell-i \rangle}.\]
From this and what proved above we deduce that:
\[m_i(V)\geq k^{\langle \ell-1,i-\ell\rangle}.\]
From the above argument and the discussion following Proposition \ref{mainlinear}, it is clear that a piecewise lexsegment monomial space $W\subset S_d$ with $m_i(W)=k^{\langle \ell-1,i-\ell\rangle} \ \forall \ i=1,\ldots ,\ell$ must exist. We have $V\subset W$ by Lemma \ref{keyextremal2}, so we get also the inequality:
\[m_i(V)\leq k^{\langle \ell-1,i-\ell\rangle}.\]
\end{proof}

We introduce the function $\TT:\NN^r\rightarrow \NN^r$ such that $\TT(\vv)=(v_1, \ v_1+v_2,\ldots , \ v_1+v_2+\ldots +v_r)$, where $\vv=(v_1,\ldots ,v_r)$. Furthermore, we define $\SS_q(\vv)$ as the last entry of $\TT^q(\vv)$.

%
%Let $\vv=(v_1,\ldots ,v_r)\in \NN^r$ be a vector of natural numbers. 
%
%Besides, for all $j=1,\ldots ,r$, denote by $\vv(j)$ the vector $\vv$ truncated to the first $j$ entries, namely $\vv(j)=(v_1,\ldots ,v_j)\in \NN^j$. We define recursively:
%\begin{itemize}
%\item[(a)] $\mathbb{S}_0(\vv)=v_r$.
%\item[(b)] $\SS_q(\vv)=\sum_{j=1}^r\SS_{q-1}(\vv(j))$ \ if $q>0$.
%\end{itemize}

\begin{remark}\label{smax1}
The significance of the above definition is the following: Let $I\subset S=K[x_1,\ldots ,x_n]$ be a stable ideal generated in one degree. One can easily show that, for all $q\in \NN$ and $i\in\{1,\ldots ,n\}$,
\[\SS_q((m_1(I),m_2(I),\ldots ,m_i(I)))=m_i(\m^q I).\]
Notice that we can also rephrase the second condition of Proposition \ref{main} as
\[\mu_{i,j}\geq \SS_1((\mu_{1,j-1},\mu_{2,j-1},\ldots ,\mu_{i,j-1})).\]
\end{remark}

\begin{example}
In the next theorem the functions $\SS_q$ will play a crucial role. Especially, using Remark \ref{smax1}, Lemma \ref{mireferee1} and Remark \ref{keyextremal1}, one has:
\begin{eqnarray*}
\SS_q((k^{\langle \ell-1,1-\ell\rangle},k^{\langle \ell-1,2-\ell\rangle},\ldots ,k^{\langle \ell-1,i-\ell\rangle})) & = & \SS_q((m_1(U(\ell ,k,d)),m_2(U(\ell ,k,d)),\ldots ,m_i(U(\ell ,k,d))))\\
& =  & m_i(\mm^qU(\ell ,k,d)))\\
& = & m_i(U(\ell ,m_{\ell}(\mm^q U(\ell ,k,d)),d+q)))\\
& = & \SS_q((k^{\langle \ell-1,1-\ell\rangle},k^{\langle \ell-1,2-\ell\rangle},\ldots ,k))^{\langle\ell-1,i-\ell\rangle}.
\end{eqnarray*}
Notice that the first time $\SS_q$ is applied to a vector in $\NN^i$, whereas the last time to a vector in $\NN^{\ell}$.
\end{example}

Let $I$ be a homogeneous ideal of $S$ and $\beta_{i,j}=\beta_{i,j}(I)$ its graded Betti numbers. 
Let the extremal Betti numbers of $I$ be
\[\beta_{i_1,i_1+j_1},\beta_{i_2,i_2+j_2},\ldots ,\beta_{i_k,i_k+j_k}.\]
Notice that $k<n$, and up to a reordering, we can assume $0<i_1<i_2<\ldots <i_k<n$ and $j_1>j_2>\ldots >j_k\geq 0$. If $I$ is a stable ideal then, exploiting the Eliahou-Kervaire formula, one can check that $\beta_{i,i+j}(I)$ is extremal if and only if $m_{i+1,j}(I)\neq 0$ and $m_{p+1,q}(I)=0$ for all $(p,q)\neq (i,j)$ such that $p\geq i$ and $q\geq j$. In this case, moreover, we have $\beta_{i,i+j}(I)=m_{i+1,j}(I)$. Before showing the main result of the paper, we introduce the following concept.

\begin{definition}
Let $\iii=(i_1,\ldots ,i_k)$ and $\jjj=(j_1,\ldots ,j_k)$ be such that $0<i_1<i_2<\ldots <i_k<n$, \ $j_1>j_2>\ldots >j_k> 0$. We say that $I\subset S$ is a {\it $(\iii,\jjj)$-lex ideal} if $I=\sum_{p=1}^k(L_p)$, where $L_p$ is a lexsegment ideal generated in degree $j_p$ in $K[x_1,\ldots ,x_{i_p+1}]$.
\end{definition}

\begin{thm}\label{extremalthm}
Let $\iii=(i_1,\ldots ,i_k)$ and $\jjj=(j_1,\ldots ,j_k)$ be such that $0<i_1<i_2<\ldots <i_k<n$ and $j_1>j_2>\ldots >j_k> 0$, and let $b_1,\ldots ,b_k$ be positive integers. For all $p=1,\ldots ,k$ let :
\[\vv^p=(b_p^{\langle i_p,-i_p \rangle},b_p^{\langle i_p,1-i_p \rangle}, \ldots ,b_p^{\langle i_p,i_{p-1}-i_p \rangle})\in \NN^{i_{p-1}+1}.\] 
If $K$ has characteristic $0$, then the following are equivalent:
\begin{itemize}
\item[{\em (i)}] There is a homogeneous ideal $I\subset S$ with extremal Betti numbers $\beta_{i_p,i_p+j_p}(I)=b_p$ for all $p=1,\ldots ,k$.
\item[{\em (ii)}] There is a strongly stable ideal $I\subset S$ with extremal Betti numbers $\beta_{i_p,i_p+j_p}(I)=b_p$ for all $p=1,\ldots ,k$.
\item[{\em (iii)}] $b_k\leq \binom{i_k+j_k-1}{i_k}$ \ and \ $\SS_{j_p-j_{p+1}}(\vv^{p+1})+b_p\leq \binom{i_p+j_p-1}{i_p}$ for all $p=1,\ldots ,k-1$.
\item[{\em (iv)}] There is an $(\iii,\jjj)$-lex ideal $I\subset S$ with extremal Betti numbers $\beta_{i_p,i_p+j_p}(I)=b_p$ for all $p=1,\ldots ,k$.
\end{itemize}
\end{thm}
\begin{proof}
(i) $\iff$ (ii) follows by by \cite[Theorem 1.6]{BCP}. (iv) $\implies$ (i) is obvious.

(ii) $\implies$ (iii). By what said before the theorem, we can replace $\beta_{i_p,i_p+j_p}(I)$ by $m_{i_p+1,j_p}(I)$ with $m_{r+1,s}(I)=0$ for all $(r,s)\neq (i_p,j_p)$ such that $r\geq i_p$ and $s\geq j_p$. Since $m_{i_k+1,j_k}(I)=b_k$, we have
\[b_k\leq \binom{i_k+j_k-1}{i_k}.\]
We must have that:
\begin{eqnarray*}
m_{i_{k-1}+1}\bigl(\mm^{j_{k-1}-j_k}(I_{\langle j_k\rangle})\bigr)+b_{k-1} & = & |\{\mbox{monomials $u\in I_{\langle j_k\rangle}\cap S_{j_{k-1}}$ with $m(u)=i_{k-1}+1$}\}|\\
& & + |\{\mbox{monomials $u\in I_{j_{k-1}}\setminus I_{\langle j_{k-1}-1\rangle}$ with $m(u)=i_{k-1}+1$}\}|\\
& \leq & |\{\mbox{monomials $u\in S_{j_{k-1}}$ with $m(u)=i_{k-1}+1$}\}|\\
& = & \binom{i_{k-1}+j_{k-1}-1}{i_{k-1}} 
\end{eqnarray*}
From the discussion before the theorem, we also have:
\[m_i\bigl(I_{\langle j_k \rangle}\bigr)\geq b_k^{\langle i_k,i-i_k-1\rangle} \ \ \forall \ i\leq i_k.\]
We eventually get:
\[m_{i_{k-1}+1}\bigl(\mm^{j_{k-1}-j_k}(I_{\langle j_k\rangle})\bigr)\geq \SS_{j_{k-1}-j_{k}}(\vv^{k}). \]
Putting together the above inequalities we obtain, for $p=k-1$,
\[\SS_{j_p-j_{p+1}}(\vv^{p+1})+b_p\leq \binom{i_p+j_p-1}{i_p},\]
and we can go on in the same way to show this for all $p=1,\ldots ,k-1$.

(iii) $\implies$ (iv). If $b_k\leq \binom{i_k+j_k-1}{i_k}$, then we can form $U(i_k+1,b_k,j_k)$. Let us call ${}^kI=U(i_k+1,b_k,j_k)$. We have that:
\[m_{i_{k-1}+1}\biggl(({}^kI)_{\langle j_{k-1}\rangle}\biggr) = \SS_{j_{k-1}-j_{k}}(\vv^{k}). \]
From Remark \ref{keyextremal1}, we deduce that 
\[({}^kI)_{\langle j_{k-1}\rangle}\cap K[x_1,\ldots ,x_{i_{k-1}+1}]=U(i_{k-1}+1,\SS_{j_{k-1}-j_{k}}(\vv^{k}),j_{k-1})\cap K[x_1,\ldots ,x_{i_{k-1}+1}].\] 
By the assumed numerical conditions, $U(i_{k-1}+1,\SS_{j_{k-1}-j_{k}}(\vv^{k})+b_{k-1},j_{k-1})$ exists and contains exactly $b_{k-1}$ new monomials $u$ such that $m(u)=i_{k-1}+1$. Therefore set:
\[{}^{k-1}I'=(U(i_{k-1}+1,\SS_{j_{k-1}-j_{k}}(\vv^{k})+b_{k-1},j_{k-1})).\]
and
\[{}^{k-1}I={}^kI+{}^{k-1}I'.\]
By construction ${}^{k-1}I$ is a $((i_{k-1},i_k),(j_{k-1},j_k))$-lex ideal with extremal Betti numbers $\beta_{i_{k-1},i_{k-1}+j_{k-1}}({}^{k-1}I)=b_{k-1}$ and $\beta_{i_k,i_k+j_k}({}^{k-1}I)=b_k$. Keeping on with the recursion we will end up with the desired $(\iii,\jjj)$-lex ideal $I={}^1I$.
\end{proof}

\begin{remark}
A different characterization of the extremal Betti numbers was also given by Crupi and Utano in \cite{CU}.
\end{remark}

\begin{remark}
For the reader who likes more the language of algebraic geometry, Theorem \ref{extremalthm} can be used in the following setting: Let $X\subset \PP^{n-1}$ be a projective scheme over a field of characteristic $0$ and $\mathcal{I}_X$ its ideal sheaf. Then, by the graded version of the Grothendieck's local duality, $\beta_{i,i+d}$ is an extremal Betti number of the ideal $\bigoplus_{m\in\NN}\Gamma(X,\mathcal{I}_X(m))\subset S$ if and only if, setting $p=n-i-1$ and $q=d-1$:
\begin{compactitem}
\item[(1)] $p\geq 1$;
\item[(2)] $\dim_K(H^p(X,\mathcal{I}_X(q-p)))=\beta_{i,i+d}\neq 0$.
\item[(3)] $H^r(X,\mathcal{I}_X(s-r))=0$ for all $(r,s)\neq (p,q)$ with $1\leq r\leq p$ and $s\geq q$.
\end{compactitem}
%%
%For the reader who likes more the language of algebraic geometry, Theorem \ref{extremalthm} rephrases as follows: Let $X\subset \PP^n$ be a projective scheme over a field of characteristic $0$. Take the special pairs $(i,j)$, with $1<i<n$, such that:
%\[H^{i-1}(X,\mathcal{O}_X(j-i))\neq 0 \mbox{ \ \ \ \ and \ \ \ \ }H^{h-1}(X,\mathcal{O}_X(k-h))= 0 \ \ \ \forall \ \ 2\leq h\leq i, \ k\geq j, \ (h,k)\neq (i,j).\] 
%We can order such special pairs like $(i_1,j_1), \ldots ,(i_k,j_k)$ where $n>i_1>\ldots >i_k>1$ and $j_1>\ldots >j_k$. Let $b_1,\ldots ,b_k$ be positive integers. With the notation of Theorem \ref{extremalthm}, we have that there exists a projective scheme $X\subset \PP^n$ with special pairs $(i_1,j_1), \ldots ,(i_k,j_k)$ and
%\[\dim_K H^{i_p-1}(X,\mathcal{O}_X(j_p-i_p))=b_p \ \ \forall \ p=1,\ldots ,k\]
%if and only if 
%\[b_k\leq \binom{n-i_k+j_k}{n-i_k} \mbox{ \ \ and \ \ } \SS_{j_p-j_{p+1}}(\vv^{p+1})+b_p\leq \binom{n-i_p+j_p}{n-i_p} \ \ \forall \ p=1,\ldots ,k-1.\]
\end{remark}

\begin{example}
Let us consider the following Betti table:
\[\left(
                   \begin{array}{ccccccc}
                     * & * & * & * & * & * & \cdots \\
                     * & * & * & a & 0 & 0 & \cdots \\
                     * & * & * & 0 & 0 & 0 & \cdots \\
                     * & * & b & 0 & 0 & 0 & \cdots \\
                     0 & 0 & 0 & 0 & 0 & 0 & \cdots
                   \end{array}
                 \right).\]

\vspace{1mm}

Theorem \ref{extremalthm} implies that there exists a homogeneous ideal in a polynomial ring (of characteristic $0$) whose Betti table looks like the above one (where $a$ and $b$ are extremal) if and only if we are in one of the following cases:
\begin{itemize}
\item[(i)] $a=2$ and $b=1,2$;
\item[(ii)] $a=1$ and $b=1,2,3,4$.
\end{itemize}
In fact, we have $b=\beta_{2,6}$ and $a=\beta_{3,5}$. Theorem \ref{extremalthm} implies $a\leq 4$.

If $a=2$, then the vector $\vv^2\in\NN^3$ is:
\[\vv^2=(1,2,2).\]
Therefore $\SS_2(\vv^2)=8$, and Theorem \ref{extremalthm} gives $8+b\leq 10$. So we get $b=1,2$ as desired.

If $a=1$, then the vector $\vv^2\in\NN^3$ is:
\[\vv^2=(1,1,1).\]
So $\SS_2(\vv^2)=6$, and Theorem \ref{extremalthm} yields $b=1,2,3,4$ as desired.

Eventually, if $a>2$, a positive integer $b$ satisfying the conditions of Theorem \ref{extremalthm} does not exist.
\end{example}

\end{document}